\documentclass{amsart}
\usepackage{graphicx}
\usepackage{amssymb}
\usepackage{amsmath}

\usepackage[all]{xy}
\xyoption{2cell}

\newcommand{\V}{{\mathcal V}}
\newcommand{\Un}{{\mathcal U}}

\newcommand{\proj}{\mathbb P}

\DeclareMathOperator{\loc}{\mathrm{Locus}}
\DeclareMathOperator{\cloc}{\mathrm{ChLocus}}
\DeclareMathOperator{\ch}{\mathrm{Chain}}
\newcommand{\ratcurves}{\textrm{Ratcurves}^n(X)}

\newcommand{\cone}{\textrm{NE}}
\DeclareMathOperator{\cycl}{N_1}
\DeclareMathOperator{\pic}{Pic}

\newcommand{\sa}{\mspace{1mu}}
\newcommand{\W}{{\mathcal{W}}}
\newcommand{\conx}[1]{\cone\,(#1,X)}
\newcommand{\cycx}[1]{\cycl(#1,X)}

\newcommand{\rc}[2]{#1 \xymatrix{\ar@{-->}[r] & }{#2}}

\newtheorem{theorem}{Theorem}[section]
\newtheorem{lemma}[theorem]{Lemma}
\newtheorem{conjecture}[theorem]{Conjecture}
\newtheorem{proposition}[theorem]{Proposition}
\newtheorem{corollary}[theorem]{Corollary}

\theoremstyle{definition}
\newtheorem{definition}[theorem]{Definition}

\theoremstyle{remark}
\newtheorem{remark}[theorem]{Remark}
\newtheorem{construction}[theorem]{Construction}

\begin{document}
\title[Bounds on the Picard number of Fano manifolds]{Rational curves and bounds on the Picard number of Fano manifolds}

\author{Carla Novelli}
\address{Dipartimento di Matematica,\newline  ``F. Enriques'', Universit\`a di Milano,\newline via C. Saldini 50, I-20133 Milano}
\curraddr{Dipartimento di Matematica ``F.Casorati'', \newline Universit\`a di Pavia,\newline via Ferrata 1, \newline I-27100 Pavia}
\email{carla.novelli@unipv.it}
\author{Gianluca Occhetta}
\address{Dipartimento di Matematica,\newline Universit\`a di Trento, \newline via Sommarive 14,\newline I-38050 Povo (TN)}
\email{gianluca.occhetta@unitn.it}

\subjclass[2000]{ 14J45, 14E30}

\maketitle

\begin{abstract}
We prove that Generalized Mukai Conjecture holds for Fano manifolds $X$ of
pseudoindex $i_X \ge (\dim X +3)/3$. We also give different proofs of the conjecture for
Fano fourfolds and fivefolds.
\end{abstract}

\section{Introduction}

Let $X$ be a Fano manifold, {\em i.e.} a smooth complex projective variety whose
anticanonical bundle $-K_X$ is ample.  The {\em index} of a Fano manifold $X$ is defined as
$$r_X := \max \{m \in \mathbb N\ |-K_X = mL  {\mathrm{\ for \ some \ line \ bundle \ }} L  \},$$
while the {\em pseudoindex} of $X$ is defined as
$$i_X := \min \{m \in \mathbb N\ |-K_X \cdot C = m {\mathrm{\ for \ some \ rational \ curve \ }} C \subset X \}.$$
We denote by $\rho_X$ the Picard number of $X$,
{\em i.e.} the dimension of the $\mathbb R$-vector space $\cycl(X)$ of $1$-cycles modulo numerical
equivalence.\par
\medskip
In 1988, Mukai \cite{Kata} proposed the following conjecture:\par
\begin{conjecture} \label{M}
Let $X$ be a Fano manifold of dimension $n$. Then $\rho_X(r_X - 1) \le n$,
with equality if and only if $X = (\proj^{r_X-1})^{\rho_X}$.
\end{conjecture}

The first step towards the conjecture was made in 1990 by Wi\'sniewski; in \cite{Wimu}, where the notion
of pseudoindex was introduced, he proved that  if $i_X > ({\dim X+2})/{2}$ then $\rho_X =1$;
moreover, if $r_X = (\dim X+2)/{2}$ then either $\rho_X =1$ or $X = (\proj^{r_X-1})^2$.\\
The problem was reconsidered in 2002 by Bonavero, Casagrande, Debarre and Druel; in \cite{BCDD} they proposed
the following more general conjecture:
\begin{conjecture} \label{GM}
Let $X$ be a Fano manifold of dimension $n$. Then $\rho_X(i_X - 1) \le n$,
with equality if and only if $X = (\proj^{i_X-1})^{\rho_X}$.
\end{conjecture}
In \cite{BCDD} Conjecture (\ref{GM}) was proved
for Fano manifolds of dimension four (in lower dimension the result
can be read off from the classification), for homogeneous manifolds, and for toric Fano manifolds
of pseudoindex $i_X \geq ({\dim X+3})/{3}$ or dimension $\le 7$.
The toric case was completely settled later by Casagrande in \cite{Cas06}.\par
\medskip
As to the general case, in 2004, Andreatta, Chierici and Occhetta in \cite{ACO} proved Conjecture (\ref{GM})
for Fano manifolds of dimension five and for Fano manifolds of pseudoindex $i_X \ge ({\dim X+3})/{3}$
admitting a special covering family of rational curves (an unsplit family, see Definition (\ref{Rf})).
They also found some sufficient condition for the existence of such a family.\par
\medskip
In this paper we reconsider the results of \cite{ACO}, and we are able to remove the extra assumption
on the existence of the special family, proving that Conjecture (\ref{GM}) holds for
Fano manifolds $X$ of pseudoindex $i_X \ge ({\dim X+3})/{3}$; this is done in section (\ref{large}).\\
In the last section of the paper we also provide a considerably shorter and simplified
proof of Conjecture (\ref{GM}) for Fano manifolds of dimension $4$ and $5$.\par
\medskip
The key result is Theorem (\ref{pasqua}), which is based on an extension of classical 
estimates of the dimension of the locus of irreducible curves of a family through a point to the locus of
limits of this curves passing through a point (Proposition (\ref{redfl2})).\\
In order to prove this result we need to recall the construction
of the scheme $\ch(\Un)$, associated to a proper covering family $\V$ of cycles.
This is the content of section (\ref{chains}), while section (\ref{back}) contains the basic definitions about families of rational curves and their properties which are of frequent use in the paper.

\section{Families of rational curves}\label{back}



\begin{definition} \label{Rf}
A {\em family of rational curves} $V$ on $X$ is an irreducible component
of the scheme $\ratcurves$ (see \cite[Definition II.2.11]{Kob}).\\
Given a rational curve we will call a {\em family of
deformations} of that curve any irreducible component of  $\ratcurves$
containing the point parametrizing that curve.\\
We define $\loc(V)$ to be the set of points of $X$ through which there is a curve among those
parametrized by $V$; we say that $V$ is a {\em covering family} if ${\loc(V)}=X$ and that $V$ is a
{\em dominating family} if $\overline{\loc(V)}=X$.\\
By abuse of notation, given a line bundle $L \in \pic(X)$, we will denote by $L \cdot V$
the intersection number $L \cdot C$, with $C$ any curve among those
parametrized by $V$.\\
We will say that $V$ is {\em unsplit} if it is proper; clearly, an unsplit dominating family is
covering.\\
We denote by $V_x$ the subscheme of $V$ parametrizing rational curves
passing through a point $x$ and by $\loc(V_x)$ the set of points of $X$
through which there is a curve among those parametrized by $V_x$. If, for a general point $x \in \loc(V)$,
$V_x$ is proper, then we will say that the family is {\em locally unsplit}; by Mori's Bend and Break arguments,
if $V$ is a locally unsplit family, then $-K_X \cdot V \le \dim X+1$.\\
If $X$ admits  dominating families, we can choose among them one with minimal degree with respect
to a fixed ample line bundle, and we call it a {\em minimal dominating family}; such
a family is locally unsplit.
\end{definition}

\begin{definition}
Let $U$ be an open dense subset of $X$ and $\pi\colon U \to Z$ a proper
surjective morphism to a quasi-projective variety;
we say that a family of rational curves $V$ is a {\em horizontal dominating family with respect to} $\pi$
if $\loc(V)$ dominates $Z$ and curves parametrized by $V$ are not contracted by $\pi$.
If such families exist, we can choose among them  one with minimal degree with respect
to a fixed ample line bundle and we call it a {\em minimal
horizontal dominating family} with respect to $\pi$; such
a family is locally unsplit.
\end{definition}

\begin{remark} By fundamental results in \cite{Mo79}, a Fano manifold admits dominating families of rational curves;
also horizontal dominating families with respect to proper morphisms defined on an open set exist,
as proved in \cite{KoMiMo}. In the case of Fano manifolds with ``minimal'' we will mean minimal with respect to $-K_X$,
unless otherwise stated.
\end{remark}

\begin{definition}\label{CF}
We define a {\em Chow family of rational 1-cycles} $\W$ to be an irreducible
component of  $\textrm{Chow}(X)$ parametrizing rational and connected 1-cycles.\\
We define $\loc(\W)$ to be the set of points of $X$ through which there is a cycle among those
parametrized by $\W$; notice that $\loc(\W)$ is a closed subset of $X$ (\cite[II.2.3]{Kob}).
We say that $\W$ is a {\em covering family} if $\loc(\W)=X$.\\
If $V$ is a family of rational curves, the closure of the image of
$V$ in $\textrm{Chow}(X)$, denoted by $\V$, is called the {\em Chow family associated to} $V$.
\end{definition}

\begin{remark}
If $V$ is proper, {\em i.e.} if the family is unsplit, then $V$ corresponds to the normalization
of the associated Chow family $\V$.
\end{remark}

\begin{definition}
Let $V$ be a family of rational curves and let $\V$ be the associated Chow family. We say that
$V$ (and also $\V$) is {\em quasi-unsplit} if every component of any reducible cycle parametrized by 
$\V$ has numerical class proportional to the numerical class of a curve parametrized by $V$.
\end{definition}

\begin{definition}
Let $V^1, \dots, V^k$ be families of rational curves on $X$ and $Y \subset X$.\\
We define $\loc(V^1)_Y$ to be the set of points $x \in X$ such that there exists
a curve $C$ among those parametrized by $V^1$ with
$C \cap Y \not = \emptyset$ and $x \in C$. We inductively define
$\loc(V^1, \dots, V^k)_Y := \loc(V^k)_{\loc(V^1, \dots,V^{k-1})_Y}$.
Notice that, by this definition, we have $\loc(V)_x=\loc(V_x)$.
Analogously we define $\loc(\W^1, \dots, \W^k)_Y$  for Chow families $\W^1, \dots, \W^k$ of rational 1-cycles.
\end{definition}

{\bf Notation}: If $\Gamma$ is a $1$-cycle, then we will denote by $[\Gamma]$ its numerical equivalence class
in $\cycl(X)$; if $V$ is a family of rational curves, we will denote by $[V]$ the numerical equivalence class
of any curve among those parametrized by $V$.\\
If $Y \subset X$, we will denote by $\cycx{Y} \subseteq \cycl(X)$ the vector subspace
generated by numerical classes of curves of $X$ contained in $Y$; moreover, we will denote 
by  $\conx{Y} \subseteq \cone(X)$
the subcone generated by numerical classes of curves of $X$ contained in~$Y$. We will
denote by $\langle \dots \rangle$ the linear span.\par
\medskip
We will make frequent use of the following dimensional estimates:

\begin{proposition} (\cite[IV.2.6]{Kob}\label{iowifam})
Let $V$ be a family of rational curves on $X$ and $x \in \loc(V)$ a point such that every component of $V_x$ is proper.
Then
  \begin{itemize}
       \item[(a)] $\dim \loc(V)+\dim \loc(V_x) \ge \dim X  -K_X \cdot V -1$;
       \item[(b)] every irreducible component of $\loc(V_x)$ has dimension $\ge -K_X \cdot V -1$.
    \end{itemize}
\end{proposition}



\begin{definition} We say that $k$ quasi-unsplit families $V^1, \dots, V^k$ are numerically independent
if in $\cycl(X)$ we have $\dim \langle [V^1], \dots, [V^k]\rangle =k$.
\end{definition}

\begin{lemma} (Cf. \cite[Lemma 5.4]{ACO}) \label{locy}
Let $Y \subset X$ be an irreducible closed subset and  $V^1, \dots, V^k$ numerically independent
unsplit families of rational curves such that \linebreak $\langle [V^1], \dots, [V^k]\rangle \cap \conx{Y}={\underline 0}$.
Then either $\loc(V^1, \ldots,V^k)_Y=\emptyset$ or
      $$\dim \loc(V^1, \ldots, V^k)_Y \ge \dim Y +\sum -K_X  \cdot V^i -k.$$
\end{lemma}

\begin{remark}
As pointed out by the referee, we need to assume the irreducibility of $Y$; otherwise
in the statement we have to replace $\dim Y$ with $\dim Y_0$ where $Y_0$ is
an irreducible component of $Y$ of minimal dimension.
\end{remark}

A key fact underlying our strategy to obtain bounds on the Picard number, based on \cite[Proposition II.4.19]{Kob},
is the following:

\begin{lemma}\label{numeq} (\cite[Lemma 4.1]{ACO})
Let $Y \subset X$ be a closed subset, $\V$ a Chow family of rational $1$-cycles. Then every curve
contained in $\loc(\V)_Y$ is numerically equivalent to a linear combination with rational
coefficients of a curve contained in $Y$ and of irreducible components of cycles parametrized
by $\V$ which meet $Y$.
\end{lemma}

\begin{corollary}\label{numcor} Let $V^1$ be a locally unsplit family of rational curves,
and $V^2, \dots, V^k$ unsplit families of rational curves. Then, for a general $x \in \loc(V^1)$,
\begin{enumerate}
\item[(a)] $\cycx{\loc(V^1)_x} = \langle [V^1] \rangle$;
\item[(b)] either $\loc(V^1, \dots, V^k)_x= \emptyset$ or 
$\cycx{\loc(V^1, \dots, V^k)_x} = \langle [V^1], \dots, [V^k] \rangle$.
\end{enumerate}
\end{corollary}

\section{Chains of rational curves}\label{chains}

Let $X$ be a smooth complex projective variety.
Let $V$ be a dominating family of rational curves on $X$ and denote by $\V$
the associated Chow family, with universal family $\Un$:

$$
\xymatrix{\Un  \ar[r]^(.50){q} \ar[d]_{p} & X\\
 \V & & }
$$

\begin{definition}
Let $Y \subset X$ be a closed subset; define $\cloc_m(\V)_Y$
to be the set of points $x \in X$ such that there exist cycles $\Gamma_1, \dots, \Gamma_m$
with the following properties:
    \begin{itemize}
       \item $\Gamma_i$ belongs to the family $\V$;
       \item $\Gamma_i \cap \Gamma_{i+1} \not = \emptyset$;
       \item $\Gamma_1 \cap Y \not = \emptyset$ and $x \in \Gamma_m$,
    \end{itemize}
{\em i.e.} $\cloc_m(\V)_Y=\loc(\V, \dots, \V)_Y$, with $\V$ appearing $m$ times, is the set of points that can be
joined to $Y$ by a connected chain  of at most $m$ cycles belonging to the family $\V$.\\
Considering among cycles parametrized by $\V$ only irreducible ones, 
in the same way one can define $\cloc_m(V)_Y$.

\end{definition}

Define a relation of {\em rational connectedness with respect to
$\V$} on $X$ in the following way: two points $x$ and $y$ of $X$ are in rc$(\V)$-relation if
there exists a chain of cycles in $\V$ which joins
$x$ and $y$, {\em i.e.} if $y \in \cloc_m(\V)_x$ for some $m$.
In particular,  $X$ is {\em $rc(\V)$-connected} if for some $m$ we have
$X=\cloc_m(\V)_x$.\par
\medskip
The family $\V$ defines a proper prerelation in the sense of \cite[Definition IV.4.6]{Kob};
to this prerelation it is associated a proper proalgebraic relation $\textrm{Chain}({\mathcal U}\sa)$ (see
\cite[Theorem IV.4.8]{Kob}) and the rc$(\V)$-relation just defined is nothing but the
set theoretic relation $\langle \mspace{1mu}{\mathcal U} \mspace{1mu}\rangle$ associated to
$\textrm{Chain}({\mathcal U}\sa)$.
We briefly  recall this construction for the reader's convenience. See \cite[IV.4]{Kob} or \cite[Appendix]{Camor}
for details.\\
Define $\ch_1(\V)$ to be the fiber product $\Un \times_\V \Un$, with projections $q_1$ and $q_2$ on $X$, which
give rise to a morphism $q_1 \times q_2 \colon \ch_1(\V) \longrightarrow  X \times X$.\\
Denoting by $\pi_i \colon (\ch_1(\V))^N \to \ch_1(\V)$ the projection onto the i-th factor and by $q_{1,i}$
(respectively $q_{2,i}$) the composition of $q_1$ (respectively $q_2$) with $\pi_i$,
inductively define
$\ch_{m+1}(\V):= \ch_{1}(\V)\times_X \ch_{m}(\V)$, as in the following diagram
$$
\xymatrix{\ch_{m+1}(\V) \ar[d] \ar[r]& \ch_m(\V) \ar[d]^{q_{1,1}}  \\
\ch_1(\V) \ar[r]_{q_2} & X }
$$
Finally set $\ch(\Un) := \bigcup_m \ch_m(\Un)$.\par
\medskip
With this language $x$ and $y$ are rc$(\V)$-equivalent if, for some $m$, 
the point $(x,y)$ is in the image of $q_{1,1} \times q_{2,m} \colon \ch_m(\V) \longrightarrow X \times X$.
The variety $X$ is then rc$(\V)$-connected  if
for some $m$ the morphism $q_{1,1} \times q_{2,m} \colon \ch_m(\V) \longrightarrow X \times X$
is dominant (hence onto, by the properness of $\V$).\par
\medskip
To the proper prerelation defined by $\V$ it is associated a fibration, which we will call the
{\em rc$(\V)$-fibration}:

\begin{theorem}(\cite[IV.4.16]{Kob}, Cf. \cite{Cam81}) Let $X$ be a normal and proper variety and $\V$
a proper prerelation; then there exists an open subvariety $X^0 \subset X$ and a proper morphism with
connected fibers $\pi\colon X^0 \to Z^0$ such that
\begin{itemize}
\item  $\langle \mspace{1mu} {\mathcal U}\mspace{1mu} \rangle$ restricts to an equivalence relation on $X^0$;
\item $\pi^{-1}(z)$ is a  $\langle \mspace{1mu} {\mathcal U} \mspace{1mu}\rangle$-equivalence class for every $z \in Z^0$;
\item $\forall\, z\in Z^0$ and  $\forall\, x,y \in \pi^{-1}(z)$, $x \in \cloc_m(\V)_y$ with
$m \le 2^{\dim X -\dim Z}-1$.
\end{itemize}
\end{theorem}

Clearly $X$ is rc$(\V)$-connected if and only if $\dim Z^0=0$.

\begin{proposition}\label{redfl}
Let $V$ be a minimal dominating family of rational curves and denote by $\V$ the associated Chow family.
Assume that $\dim \loc(V)_x \ge s$, for a general $x \in X$ and some integer $s$; then
for every $x \in X$ every irreducible component of $\loc(\V)_x$ has dimension $\ge s$.
\end{proposition}

\begin{proof} Consider the morphism  $q_1 \times q_2 \colon \ch_1(\V) \longrightarrow  X \times X$; by \cite[Lemme 2]{Cam81} (or \cite[Lemma 1.14]{Camor}) we know that  $\ch_1(\V)$ is irreducible. Denote by $\mathcal C^1$ the image
$(q_1 \times q_2)(\ch_1(\V)) \subset X \times X$.\\
Let $p: \mathcal C^1 \to X$ be the restriction of the first projection; the inverse image of a point $x_0$ via $p$ consists
of the points which belong to a cycle in $\V$ containing $x_0$, hence $p^{-1}(x_0) = \loc(\V)_{x_0}$. By the minimality assumption, through a general point  $x \in X$ there are no reducible cycles, hence $\dim p^{-1}(x) \ge s$. The statement now follows by the semicontinuity of the local dimension of a fiber (\cite[Corollary 3, pag. 51]{Mu}), which ensures that
the dimension of every irreducible component of every fiber of $p$ has dimension $\ge s$. 
\end{proof}

\begin{corollary}\label{redfl2}
Let $V$ be a minimal dominating family of rational curves and denote by $\V$ the associated Chow family.
Then every  irreducible component of $\loc(\V)_x$ has dimension $\ge -K_X \cdot V -1$.
\end{corollary}

Given $\V^1, \dots, \V^k$ Chow families of rational 1-cycles, it is possible to define, as above, a relation of
rc$(\V^1, \dots, \V^k)$-connectedness,  to which it is associated a fibration, which we will call
rc$(\V^1, \dots, \V^k)$-fibration. The variety $X$ will be called  {\em rc$(\V^1, \dots, \V^k)$-connected} if
the target of the fibration is a point.\par
\medskip
For such varieties we have the following application of Lemma (\ref{numeq}):

\begin{proposition} (Cf. \cite[Corollary 4.4]{ACO})\label{rhobound}
If $X$ is rationally connected with respect to some Chow families of rational $1$-cycles
$\V^1, \dots, \V^k$, then $\cycl(X)$ is generated by the classes of irreducible components of cycles
in $\V^1, \dots, \V^k$.\\
In particular, if $\V^1, \dots, \V^k$ are quasi-unsplit families, then $\rho_X \le k$ and equality
holds if and only if $\V^1, \dots, \V^k$ are numerically independent.
\end{proposition}

A straightforward consequence of the above proposition is the following:

\begin{corollary}\label{rhoboundcor}
If $X$ is rationally connected with respect to Chow families of rational $1$-cycles
$\V^1, \dots, \V^k$ and $D$ is an effective divisor, then $D$ cannot be trivial on every irreducible component
of every cycle parametrized by $\V^1, \dots, \V^k$.
\end{corollary}

\section{Large pseudoindex}\label{large}

In this section we will prove a bound on the Picard number of Fano manifolds which are rationally
connected with respect to a special Chow family. Then
we will show that Conjecture (\ref{GM}) holds for Fano manifolds $X$ of pseudoindex
$i_X \ge (\dim X+3)/{3}$.\\

We start with a technical result:

\begin{lemma}\label{freq} Let $X$ be a Fano manifold of pseudoindex $i_X$, let $Y \subset X$ be a closed irreducible subset
of dimension $\dim Y > \dim X - i_X$
and let $W$ be an unsplit non dominating family of rational curves such that $[W] \not \in \conx{Y}$.\\
Then $\loc(W) \cap Y = \emptyset$.
\end{lemma}

\begin{proof} If the intersection were nonempty, by Lemma (\ref{locy}) we would
have 
$$\dim \loc(W)_Y \ge \dim Y -K_X \cdot W -1 > \dim X-1,$$ 
so $W$ would be a dominating family,
a contradiction.  
\end{proof}

\begin{theorem}\label{pasqua} Let $X$ be a Fano manifold of Picard number $\rho_X$ 
and pseudoindex $i_X$, and let $V$ be a minimal dominating family
of rational curves for $X$. Assume that $X$ is rc$(\V)$-connected and that 
$3i_X > -K_X \cdot V > \dim X+1-i_X$.
Then $\rho_X=1$.
\end{theorem}

\begin{proof}
Since $X$ is rc$(\V)$-connected, for some integer $m$ the morphism 
$\ch_m(\V) \longrightarrow X \times X$ is onto; equivalently, $X=\cloc_m(\V)_x$ for every
$x \in X$. Let $x$ be a general point; we will show that every irreducible component of a $\V$-cycle in a connected $m$-chain passing though $x$ is numerically proportional to $V$. The statement will then follow by repeated applications of Lemma (\ref{numeq}).\\
Since $-K_X \cdot V < 3i_X$, any reducible $\V$-cycle $\Gamma$ has two irreducible components,
hence either both of them are numerically proportional to $V$ or neither of them is numerically proportional to $V$.\\
Assume by contradiction that there exist $m$-chains through $x$, $\Gamma_1 \cup \Gamma_2 \cup \dots \cup \Gamma_m$, with $x \in \Gamma_1$ and $\Gamma_i \cap \Gamma_{i+1} \not = \emptyset$, such that, 
for some $j \in \{1, \dots, m\}$ the irreducible components $\Gamma_j^1$ and $\Gamma_j^2$ of $\Gamma_j$
are not numerically proportional to $\V$.\\
Let $j_0 \in \{1, \dots, m\}$ be the minimum integer for which such a chain exists; by the generality of $x$ we have $j_0 \ge 2$.
If $j_0=2$ set $x_1=x$, otherwise let $x_1$ be a point in $\Gamma_{j_0-1} \cap \Gamma_{j_0 -2}$. Since
$\Gamma_{j_0-1} \subset \loc (\V)_{x_1}$ there is  an irreducible component $Y$ of $\loc(V)_{x_1}$
which meets $\Gamma_{j_0}$.\\ 
By Corollary (\ref{redfl2}) we have $\dim Y \ge -K_X \cdot V -1 > \dim X -i_X$; moreover, since $j_0$ was minimal,
every cycle parametrized by $\V$ passing through $x_1$ is numerically proportional to $\V$, hence
$\cycx{Y}=\langle [V] \rangle$ by Lemma (\ref{numeq}).\\
Let $\gamma$ be a component of $\Gamma_{j_0}$ meeting $Y$.
Denote by $W$ a family of deformations of $\gamma$; then the
family $W$ is unsplit, as $-K_X \cdot V < 3i_X$ and it is not dominating, by the minimality of $V$.
We now get the desired contradiction by Lemma (\ref{freq}).  
\end{proof}

\begin{construction}\label{kfam} Let $X$ be a Fano manifold; let
$V^1$ be a minimal dominating family of rational curves on $X$ and consider the associated Chow family $\V^1$.\\
If $X$ is not rc$(\V^1)$-connected, let $V^2$ be a minimal
horizontal dominating family with respect to the rc$(\V^1)$-fibration,
$\pi_{1}\colon X^{} \xymatrix{\ar@{-->}[r]^{}&} Z^{1}$. 
If $X$ is not rc$(\V^1, \V^2)$-connected,
we denote by $V^3$ a minimal horizontal dominating family with respect to the
the rc$(\V^1,\V^2)$-fibration, $\pi_{2}\colon X^{} \xymatrix{\ar@{-->}[r]^{}&} Z^{2}$, and so on.
Since $\dim Z^{i+1} < \dim Z^i$, for some integer $k$ we have that $X$ is
rc$(\V^1, \dots, \V^k)$-connected.\\
Notice that, by construction, the families $V^1, \dots, V^k$ are numerically independent.
\end{construction}

\begin{lemma}\label{kfamprop}
Let $X$ be a Fano manifold of pseudoindex $i_X \ge 2$ and let $V^1, \dots, V^k$ be families
of rational curves as in Construction (\ref{kfam}). Then
$$\sum_{i=1}^k(-K_X \cdot V^i -1) \le \dim X.$$
In particular, $k (i_X -1) \le \dim X$, and equality holds if and only if $X = (\proj^{i_X-1})^k$.
\end{lemma}

\begin{proof}
In Construction (\ref{kfam}) at the $i$-th step, denoted by $x_i$ a general point in $\loc(V^i)$, the dimension of the 
quotient drops at least by $\dim \loc(V^i)_{x_i}$, which,
by part (b) of Proposition (\ref{iowifam}), is greater than or equal to $-K_X \cdot V^i -1$.
It follows that
$$\dim X \ge \sum_{i=1}^k \dim \loc(V^i)_{x_i} \ge\sum_{i=1}^k(-K_X \cdot V^i -1) \ge ki_X -k = k (i_X-1).$$
If $\dim X = k (i_X-1)$, then for any $i$ we have $-K_X \cdot V^i= i_X$,
so $V^i$ is an unsplit family and $\dim \loc(V^i)_{x_i} = i_X -1$, hence the family $V^i$ is covering
by part (a) of Proposition (\ref{iowifam}). We can now apply \cite[Theorem 1]{Op} to conclude.  
\end{proof} 

\begin{theorem}\label{terzi}
Let $X$ be a Fano manifold of Picard number $\rho_X$ and pseudoindex $i_X \ge (\dim X+3)/{3}$. Then
$\rho_X (i_X-1) \leq \dim X$ and equality holds if and only if $X=(\proj^{i_X-1})^{\rho_X}$.
\end{theorem}

\begin{proof}
Let $V^1, \dots, V^k$ be families of rational curves such as in Construction (\ref{kfam});
by Lemma (\ref{kfamprop}) we have that $k (i_X -1) \le \dim X$.
If for some $j$ the family $V^j$ is not unsplit we have $-K_X \cdot V^j \ge 2i_X$, so, again by Lemma (\ref{kfamprop}), 
this can happen for at most one $j$ and implies $k = 1$.\par
\medskip
If all the families $V^i$ are unsplit, then we have $\rho_X = k$ by Proposition (\ref{rhobound}).
Moreover, if $k(i_X-1)=\dim X$, by Lemma (\ref{kfamprop}) we have $X=(\proj^{i_X-1})^{\rho_X}$.\par
\medskip
We can thus assume that $V^1$ is not unsplit and $X$ is rc$(\V^1)$-connected.\\
By the minimality of $V^1$ we have $-K_X \cdot V^1 \le \dim X+1 <3i_X$; on the other hand, since $V^1$ is not unsplit,
we have
$$-K_X \cdot V^1 \ge 2i_X \ge 2\,\,\frac{\dim X+3}{3} > 2\;\frac{\dim X}{3} \ge \dim X+1 -i_X,$$
so we can apply Theorem (\ref{pasqua}) to conclude.  
\end{proof}

\section{Low dimensions}

In this section we will present different proofs of Conjecture (\ref{GM}) for Fano manifolds
of dimension four and five, which are simpler and shorter than the original ones.

\begin{theorem}\label{four}
Let $X$ be a Fano manifold of Picard number $\rho_X$, pseudoindex $i_X$ and dimension $4$. Then
$\rho_X (i_X-1) \leq 4$.
Moreover, equality holds if and only if $X=\proj^4$, or $X=\proj^2 \times \proj^2$, or
$X=\proj^1 \times \proj^1 \times \proj^1 \times \proj^1$.
\end{theorem}

\begin{proof} Clearly we can assume $i_X \ge 2$.
Let $V^1, \dots, V^k$ be families of rational curves as in Construction (\ref{kfam});
by Lemma (\ref{kfamprop}) we get $k(i_X-1)\le 4$, hence $k \le 4$.\\
If for some $j$ the family $V^j$ is not unsplit, then $-K_X \cdot V^j \ge 2i_X \ge 4$, hence, by Lemma (\ref{kfamprop}), 
this can happen for at most one $j$ and implies $k \le 2$ and $i_X=2$.\par
\medskip
If all the families $V^i$ are unsplit, then $\rho_X =k$ by Proposition (\ref{rhobound}).
Moreover, if $k(i_X-1)=4$, we have $X=(\proj^{i_X-1})^{\rho_X}$ by Lemma (\ref{kfamprop}).\par
\smallskip
We can thus assume that one of these families, say $V^j$, is not unsplit. 
By part (a) of Corollary (\ref{numcor}), we have $\cycl(\loc(V^j)_{x_j}, X)=\langle [V^j] \rangle$
for a general point $x_j \in \loc(V^j)$.\par
\smallskip
If $j=2$, then, for a general point $x_2 \in \loc (V^2)$, we have 
$X= \loc(V^2,V^1)_{x_2}$ by Lemma (\ref{locy}). Therefore, by part (b) of Corollary (\ref{numcor}), we obtain
that $\cycl(X) = \langle [V^1], [V^2] \rangle$, so $\rho_X=2$.\par
\smallskip
Assume now that $j=1$, {\em i.e.} $V^1$ is not unsplit. We claim that $X$ is rc$(\V^1)$-connected.\\
Notice that, by the minimality of $V^1$,
we can assume that $X$ has no dominating families of rational curves of anticanonical
degree $< 2i_X =4$. If $X$ is not rc$(\V^1)$-connected, since a general fiber
of $\pi_1$ contains $\loc(V^1)_{x_1}$ which, by part (b) of Proposition (\ref{iowifam}),
has dimension at least three, then $\dim Z^1 = 1$.
By part (b) of Proposition (\ref{iowifam}), for a general point $x_2 \in \loc(V^2)$, we get
$$2 \le -K_X \cdot V^2 \le \dim \loc(V^2)_{x_2}+1 \le \dim Z^1+1 = 2,$$
hence $V^2$ has anticanonical degree $2$, and so it is dominating by part (a) of the same proposition, 
a contradiction which proves the claim.\par
\smallskip
Therefore $X$ is rc$(\V^1)$-connected and we can apply Theorem (\ref{pasqua}) to get
$\rho_X=1$.  
\end{proof}

\begin{theorem}\label{five}
Let $X$ be a Fano manifold of Picard number $\rho_X$, pseudoindex $i_X$ and dimension $5$. Then
$\rho_X (i_X-1) \leq 5$.
Moreover, equality holds if and only if either $X=\proj^5$ or 
$X=\proj^1 \times \proj^1 \times \proj^1 \times \proj^1 \times \proj^1$.
\end{theorem}

\begin{proof}
Clearly we can assume $i_X\ge 2$.
Let $V^1, \dots, V^k$ be families of rational curves as in Construction (\ref{kfam});
by Lemma (\ref{kfamprop}) we get $k(i_X-1)\le 5$, hence $k \le 5$.\\
If $V^j$ is not unsplit for some $j$, then $-K_X \cdot V^j \ge 2i_X \ge 4$, hence, by Lemma (\ref{kfamprop}) 
this can happen for at most one $j$ and implies $k \le 3$.
Notice that, if $-K_X \cdot V^j \ge \dim X+1$, then $-K_X \cdot V^j = \dim X+1$ and $k=j=1$; moreover,
if $j \not = 1$ then $i_X=2$.\par
\medskip
If all the families $V^i$ are unsplit, then $\rho_X =k$ by Proposition (\ref{rhobound}).
Moreover, if $k(i_X-1)=5$, we have $X=(\proj^{i_X-1})^{\rho_X}$ by Lemma (\ref{kfamprop}).\par
\smallskip
We can thus assume that one of these families, say $V^j$, is not unsplit. 
By part (a) of Corollary (\ref{numcor}), we have $\cycl(\loc(V^j)_{x_j}, X)=\langle [V^j] \rangle$
for a general point $x_j \in \loc(V^j)$.\par
\medskip
If $j=3$, then, for a general point $x_3 \in \loc (V^3)$, we have 
$X= \loc(V^3,V^2,V^1)_{{x_3}}$ by Lemma (\ref{locy}).
Therefore, by part (b) of Corollary (\ref{numcor}), we obtain that
$\cycl(X) = \langle [V^1], [V^2], [V^3]\rangle$, so $\rho_X=3$.\par
\medskip
Assume now that $j=2$. We claim that $X$ is rc$(V^1,\V^2)$-connected.\\
 In fact, by part (b) of Proposition (\ref{iowifam}), we have
$\dim \loc(V^2)_{x_2} \ge 3$ for a general point $x_2 \in \loc(V^2)$; therefore
a general fiber of the rc$(V^1,\V^2)$-fibration has dimension at least $\dim \loc(V^2,V^1)_{x_2}$, which
is at least four by Lemma (\ref{locy}).\\
This implies $\dim Z^2 \le 1$, and thus, if $X$ were not rc$(V^1,\V^2)$-connected, we would have
$\dim \loc(V^3)_{x_3}=1$ for a general point $x_3 \in \loc(V^3)$.
Hence, by part (b) of Proposition (\ref{iowifam}), $-K_X \cdot V^3 = 2 =i_X$, so $V^3$ would be unsplit and,
by part (a) of the same proposition, covering, against the minimality of $V^2$.\par
\smallskip

Consider an irreducible component of $\loc(V^2,V^1)_{{x_2}}$
of maximal dimension. By Lemma (\ref{locy}) this dimension is $\ge 4$, hence either 
$X=\loc(V^2,V^1)_{{x_2}}$ and $\rho_X=2$ by part (b) of Corollary (\ref{numcor}),
or this component is a divisor, call it $D$.\\
If $D \cdot V^1 > 0$ then, being $V^1$ covering, we have
$X=\loc(V^1)_D$, and $\rho_X=2$ by Lemma (\ref{numeq}) and part (b) of Corollary (\ref{numcor}).\\
Assume now that $D \cdot V^1=0$. By Lemma (\ref{freq}), $D$ cannot meet components of reducible cycles of $\V^2$
whose classes are not contained in the plane spanned by $[V^1]$ and $[V^2]$ in $N_1(X)$.
So, if there were such a reducible cycle $\Gamma =\Gamma_1 + \Gamma_2$, 
we would have $D \cdot \Gamma_i=0$, hence also
$D \cdot V^2 =0$, and $D$ would be trivial on every component of every cycle in $V^1$ and $\V^2$,
against Corollary~(\ref{rhoboundcor}).\\
It follows that the class of every reducible cycle of $\V^2$
is contained in the plane spanned by $[V^1]$ and $[V^2]$ and $\rho_X=2$ by Proposition
(\ref{rhobound}). \par 
\medskip
Finally assume that $j=1$, {\em i.e.} $V^1$ is not unsplit. Notice that, by the minimality of $V^1$,
we can assume that $X$ has no dominating families of rational curves of anticanonical
degree $< 2i_X$.\par
\smallskip
If $X$ is not rc$(\V^1)$-connected, since a general fiber
of $\pi_1$ contains $\loc(V^1)_{x_1}$ which, by part (b) of Proposition (\ref{iowifam}),
has dimension at least three, then $\dim Z^1 \le 2$.
It follows that,
by part (b) of Proposition (\ref{iowifam}),
$$-K_X \cdot V^2 \le \dim \loc(V^2)_{x_2}+1 \le \dim Z^1+1 \le 3,$$
for a general point $x_2 \in \loc(V^2)$;
hence $V^2$ has anticanonical degree $<2i_X$, so it can not be dominating. This also
implies $\dim Z^1 =2$.\par
\smallskip
For a general point $x_1 \in \loc(V^1)$, we get $\dim \loc(V^1, V^2)_{x_1} \ge 4$ by Lemma~(\ref{locy}).
Since $D:= \loc(V^1, V^2)_{x_1}$ is contained in $\loc(V^2)$ and $V^2$ 
is not dominating we have $D = \loc(V^2)$. In particular $D$ dominates $Z^1$ and so
meets a general fiber of $\pi_1$, which is $\loc(V^1)_{x_1}$,
for some $x_1$, by dimensional reasons. It follows that $X=\loc(\V^1)_D$.\\
By part (b) of Corollary (\ref{numcor}), we have $\cycx{D}= \langle [V^1], [V^2]\rangle$; hence, by Lemma (\ref{freq}),
$D$ cannot meet reducible cycles of $\V^1$
whose classes are not contained in the plane spanned by $[V^1]$ and $[V^2]$.
Therefore $\rho_X=2$ by Lemma (\ref{numeq}).\par
\smallskip
Finally we deal with the case in which $X$ is rc$(\V^1)$-connected; let $x$ be a general point.
Since $x$ is general and $V^1$ is minimal we have $\overline{\loc(V^1)_x}=\loc(V^1)_x$ and 
$\cycx{\loc(V^1)_x}=\langle [V^1] \rangle$ by part (a) of Corollary (\ref{numcor}).\\
If $\loc(V^1)_x=X$, then  $\rho_X =1$. So we can suppose that
$\dim \loc(V^1)_x < 5$, and thus, by part (b) of Proposition  (\ref{iowifam}), $-K_X \cdot V^1 < 3i_X$ 
and $i_X=2$; in particular every reducible cycle parametrized by $\V^1$ has two irreducible components.\par
\smallskip
If every irreducible component of a $\V^1$-cycle in a connected $m$-chain though $x$ is numerically proportional to $V^1$ then $\rho_X=1$ by repeated applications of Lemma (\ref{numeq}).\\
We can thus assume that there exist $m$-chains through $x$, $\Gamma_1 \cup \Gamma_2 \cup \dots \cup \Gamma_m$, with $x \in \Gamma_1$ and $\Gamma_i \cap \Gamma_{i+1} \not = \emptyset$, such that, 
for some $j \in \{1, \dots, m\}$ the irreducible components $\Gamma_j^1$ and $\Gamma_j^2$ of $\Gamma_j$
are not numerically proportional to $\V^1$.\\
Let $j_0 \in \{1, \dots, m\}$ be the minimum integer for which such a chain exists; by the generality of $x$ we have $j_0 \ge 2$.
If $j_0=2$ set $x_1=x$, otherwise let $x_1$ be a point in $\Gamma_{j_0-1} \cap \Gamma_{j_0 -2}$. Since
$\Gamma_{j_0-1} \subset \loc (\V^1)_{x_1}$ there is  an irreducible component $Y$ of $\loc(V^1)_{x_1}$
which meets $\Gamma_{j_0}$. By Corollary (\ref{redfl2}) we have $\dim Y \ge -K_X \cdot V^1 -1  \ge 3$, and,
by Lemma (\ref{numeq}), $\cycx{Y} = \langle [V^1] \rangle$.\\
Let $\gamma$ be a component of $\Gamma_{j_0}$ meeting $Y$ and
denote by $W$ a family of deformations of $\gamma$; then the
family $W$ is unsplit, as $-K_X \cdot V^1 < 3i_X$, and it is not covering, by the minimality of $V^1$.\\
By Lemma (\ref{locy}) we have $\dim \loc(W)_{Y} \ge 4$, hence $\loc(W)=\loc(W)_{Y}$; 
by part (b) of Corollary (\ref{numcor}) we get $\cycl(\loc(W)_{Y},X)= \langle [V^1], [W] \rangle$.\\
Denote by $G$ the divisor $\loc(W)$; since $G$ meets $Y$
and does not contain it, being $x$ general, we have $G \cdot V^1 >0$; therefore $X=\loc(\V^1)_G$.
Since $G \cdot V^1 >0$ we also have that, if $\Gamma_1 + \Gamma_2$
is a reducible cycle parametrized by $\V^1$, then $G \cdot \Gamma_i >0$ for at least one $i=1,2$.\\
On the other hand, in view of Lemma (\ref{freq}), $G$ must be trivial on every irreducible component of
a cycle in $\V^1$ not contained in the plane spanned by $[V^1]$ and $[W]$.
Therefore there are no such cycles, and $\rho_X=2$ by Proposition (\ref{rhobound}).  
\end{proof}

\bigskip

\noindent
\small{{\bf Acknowledgements}. 
We would like to thank the referee for his corrections and suggestions, in particular for finding a serious gap
in one of our arguments. We also thank Edoardo Ballico, Riccardo Ghiloni and Roberto Pignatelli for helpful discussions.


\begin{thebibliography}{10}
\providecommand{\url}[1]{{#1}}
\providecommand{\urlprefix}{URL }
\expandafter\ifx\csname urlstyle\endcsname\relax
  \providecommand{\doi}[1]{DOI~\discretionary{}{}{}#1}\else
  \providecommand{\doi}{DOI~\discretionary{}{}{}\begingroup
  \urlstyle{rm}\Url}\fi

\bibitem{ACO}
Andreatta, M., Chierici, E., Occhetta, G.: Generalized {M}ukai conjecture for
  special {F}ano varieties.
\newblock Cent. Eur. J. Math. \textbf{2}(2), 272--293 (2004)  
  
\bibitem{BCDD}
Bonavero, L., Casagrande, C., Debarre, O., Druel, S.: Sur une conjecture de
  {M}ukai.
\newblock Comment. Math. Helv. \textbf{78}(3), 601--626 (2003)

\bibitem{Cam81}
Campana, F.: Cor\'eduction alg\'ebrique d'un espace analytique faiblement
  {K}\"ahl\'erien compact.
\newblock Invent. Math. \textbf{63}(2), 187--223 (1981)

\bibitem{Camor}
Campana, F.: Orbifolds, special varieties and classification theory: an
  appendix.
\newblock Ann. Inst. Fourier (Grenoble) \textbf{54}(3), 631--665 (2004)


\bibitem{Cas06}
Casagrande, C.: The number of vertices of a {F}ano polytope.
\newblock Ann. Inst. Fourier (Grenoble) \textbf{56}(1), 121--130 (2006)



\bibitem{Kob}
Koll{\'a}r, J.: Rational curves on algebraic varieties, \emph{Ergebnisse der
  Mathematik und ihrer Grenzgebiete}, vol.~32.
\newblock Springer-Verlag, Berlin (1996)

\bibitem{KoMiMo}
Koll{\'a}r, J., Miyaoka, Y., Mori, S.: Rational connectedness and boundedness
  of {F}ano manifolds.
\newblock J. Differential Geom. \textbf{36}(3), 765--779 (1992)

\bibitem{Mo79}
Mori, S.: Projective manifolds with ample tangent bundles.
\newblock Ann. of Math. (2) \textbf{110}(3), 593--606 (1979)

\bibitem{Kata}
Mukai, S.: Open problems.
\newblock In: T.~{T}aniguchi foundation (ed.) Birational geometry of algebraic
  varieties. Katata (1988)

\bibitem{Mu}
Mumford, D.: The red book of varieties and schemes,
 \emph{Lecture Notes in Mathematics}, vol.~1358
\newblock Springer-Verlag, Berlin (1999)



\bibitem{Op}
Occhetta, G.: A characterization of products of projective spaces.
\newblock Canad. Math. Bull. \textbf{49}, 270--280 (2006)

\bibitem{Wimu}
Wi{\'s}niewski, J.A.: On a conjecture of {M}ukai.
\newblock Manuscripta Math. \textbf{68}(2), 135--141 (1990)


\end{thebibliography}
\end{document}